\documentclass[10pt, english,reqno]{amsart}

\usepackage{array}

\usepackage{booktabs}
\usepackage{amsmath,amssymb,enumerate}
\usepackage{mathrsfs}

\usepackage[T1]{fontenc}
\usepackage[all]{xy}

\usepackage{mathtools}
\usepackage{babel}
\usepackage{amstext}
\usepackage{amsmath}
\usepackage{amsfonts}
\usepackage{latexsym}
\usepackage{ifthen}
\usepackage{subeqnarray}

\usepackage{cases}
\usepackage[table]{xcolor}
\usepackage{longtable}

\newcolumntype{M}[1]{>{\centering\arraybackslash}m{#1}}

\usepackage{xypic}
\xyoption{all}
\usepackage{float}

\theoremstyle{plain}
\newtheorem{theorem}{Theorem}[section]
    \newtheorem{corollary}[theorem]{Corollary}
    \newtheorem{lemma}[theorem]{Lemma}
    \newtheorem{proposition}[theorem]{Proposition}
    \newtheorem{conjecture}[theorem]{Conjecture}

    \theoremstyle{definition}
    \newtheorem{definition}[theorem]{Definition}
    \newtheorem{example}[theorem]{Example}

    \newtheorem{remark}[theorem]{Remark}
    \newtheorem{question}[theorem]{Question}

\usepackage{eurosym}

\makeatletter
\ifnum\@ptsize=0  \fi \ifnum\@ptsize=2  \fi 

\title{On pseudoeffective thresholds and cohomology of twisted symmetric tensor fields on irreducible Hermitian symmetric spaces}

\author{Feng Shao}

\thanks{
Email: shaofeng16@mails.ucas.ac.cn}

\begin{document}

\thanks{
Academy of Mathematics and Systems Science, Chinese Academy of Sciences, Beijing, 100190, China and University of Chinese Academy of Sciences, Beijing, 100049, China
}

\begin{abstract}
Let $X$ be an irreducible Hermitian symmetric space of compact type (IHSS for short). In this paper, we give the irreducible decomposition of $Sym^r T_X$. As a by-product, we give a cohomological characterization of the rank of $X$. Moreover, we introduce pseudoeffective thresholds to measure the bigness of  tangent bundles of smooth complex projective varieties precisely and calculate them for irreducible Hermitian symmetric spaces of compact type.
\end{abstract}

\newpage

\maketitle

\section{Introduction}
Throughout this paper, we work over the complex number field.

Let $X$ be a smooth projective variety and let $T_X$ be its tangent bundle. By Mori's famous work, $T_X$ is ample if and only if $X$ is a projective space. It is natural to consider smooth projective varieties whose tangent bundles have weaker positivity, for example, smooth projective varieties with big tangent bundles. Recall that a line bundle $L$ is big if the rational map $\phi_m:X\dashrightarrow \mathbb{P}(H^0(X,L^{\otimes m}))$ is birational onto its image for some $m>0$ and a vector bundle $E$ is big if the tautological bundle $\mathcal{O}_{\mathbb{P}(E)}(1)$ is big on the Grothendieck projectivization $\mathbb{P}(E)$.

In the last few years, there appear several works on the bigness of the tangent bundle $T_X$. It is known that a vector bundle $E$ on $X$ is big if and only if for some (or every) ample line bundle $L$ on $X$, one has $H^0(X,Sym^r E\otimes L^{-1})\neq 0$ for some positive integer $r$. Hence by \cite[Theorem 0.1]{campana2011geometric}, the bigness of the tangent bundle $T_X$ implies the uniruledness of $X$. This fact also appears in \cite[Proposition  4.6]{greb2020canonical}. Futhermore, Hsiao \cite[Corollary 1.3]{hsiao2015a} gave a criterion for the bigness of the tangent bundle $T_X$ in terms of its section ring and as a corollary, he showed that $T_X$ is big if $X$ is a smooth projective toric variety or a partial flag variety(i.e. a rational homogeneous space $G/P$ where $G$ is of type $A$). He also asked a question whether Fano manifolds with nef tangent bundles have big tangent bundles. Greb and Wong \cite[Corollary 4.4]{greb2020canonical} showed that the affineness of canonical extensions implies the bigness of the tangent bundles and as a corollary they proved that rational homogeneous spaces have big tangent bundles (see Proposition \ref{bignessofRHS} for another proof). Hence Hsiao's question has a positive answer if the Campana-Peternell conjecture holds. Recently, H\"{o}ring, Liu and Shao (\cite{HLS22}) proved that the tangent bundle of a del Pezzo manifold $X$ of dimension 2 or 3 is big if and only if the degree of $X$ is at least 5 and the tangent bundle of a smooth hypersurface $X$ in a projective space is big if and only if the degree of $X$ is at most 2. Mallory also considered the bigness of tangent bundles of del Pezzo surfaces in \cite{Mal21}. Their results imply that the bigness of tangent bundles is a rather restrictive property.

In order to give more precise information on the bigness of tangent bundles, we introduce the following:
\begin{definition}\label{bignessdefect}
  Let $X$ be a smooth projective variety. Let $\xi=c_1(\mathcal{O}_{\mathbb{P}(T_X)}(1))$ be the tautological class on $\mathbb{P}(T_X)$ and $A$ an ample divisor on $X$. Define
\begin{center}
   $\Delta (X,A)=sup\{\varepsilon \in \mathbb{R}|\xi -\varepsilon \pi^{*}A$ is effective$\}$,
\end{center}
where $\pi:\mathbb{P}(T_X)\rightarrow X$ is the natural projection.
We call $\Delta(X,A)$ the pseudoeffective threshold of $T_X$ with respect to the ample divisor $A$.
\end{definition}

It is easy to prove that
$$
\Delta(X,A)=sup\{\frac{d}{r}|H^0(X,Sym^r T_X \otimes \mathcal{O}_X(-dA))\neq 0\}
$$
(see Proposition \ref{cohomology}). Hence $T_X$ is big if and only if $\Delta(X,A)>0$. In Section \ref{PT}, we will give some properties of the pseudoeffective threshold and present the following conjecture to give a characterization of projective spaces and hyperquadrics:

\begin{conjecture}\label{conj}
Let $X$ be a smooth projective variety and $A$ an ample divisor on $X$. Suppose that $\Delta(X,A)\geq 1$. Then $(X,A)\simeq (\mathbb{P}^n,\mathcal{O}_{\mathbb{P}^n}(1))$ or $(Q^n,\mathcal{O}_{Q^n}(1))$, where $(Q^1,\mathcal{O}_{Q^1}(1))=(\mathbb{P}^1,\mathcal{O}_{\mathbb{P}^1}(2))$.
\end{conjecture}

Note that \cite[Theorem B]{druel2013characterizations} implies that if $H^0(X,Sym^r T_X \otimes \mathcal{O}_X(-rA))\neq 0$, then $X$ must be a projective space or a hyperquadric. This conjecture states that if $sup\{\frac{d}{r}|H^0(X,Sym^r T_X \otimes \mathcal{O}_X(-dA))\neq 0\}=1$, then $X$ must be a projective space or a hyperquadric. So the conjecture is natural and we will prove it under the assumption that $X$ has Picard number 1 (See Proposition \ref{picardnumber1}).

The pseudoeffective threshold of $T_X$ also has deep connection with the psdudoeffective cone of $\mathbb{P}(T_X)$ and note that the interior of the pseudoeffective cone is the big cone. If $X$ has Picard number 1 and its Picard group is generated by the ample divisor $A$, then $\mathbb{P}(T_X)$ has Picard number 2 and the boundary of the pseudoeffective cone of $\mathbb{P}(T_X)$ consists of the extremal rays $\mathbb{R}^{\geq 0}(\pi^*A)$ and $\mathbb{R}^{\geq 0}(\xi-\Delta (X,A) \pi^{*}A)$. So the pseudoeffective cone of $\mathbb{P}(T_X)$ is uniquely determined by $\Delta(X,A)$ in this case. In this paper, we consider a special kind of smooth projective varieties with Picard number 1 and calculate the pseudoeffective thresholds of their tangent bundles. First, we establish the following result on the cohomology of twisted symmetric tensor fields on irreducible Hermitian symmetric spaces of compact type (IHSS for short).

\begin{theorem}\label{maincoh}
Let $X=G/P_k$ be an irreducible Hermitian symmetric space of compact type and let $X\hookrightarrow \mathbb{P}^N=\mathbb{P}V_{\lambda_k}^*$ be a minimal closed embedding. If $V_{\lambda_k}$ is not a self-dual $G$-module, then
$$
H^0(X,Sym^r T_X \otimes \mathcal{O}_X(-d))\neq 0 \Leftrightarrow \lfloor \frac{r}{rk(X)}\rfloor \geq d;
$$
If $V_{\lambda_k}$ is a self-dual $G$-module, then
$$
H^0(X,Sym^r T_X \otimes \mathcal{O}_X(-d))\neq 0 \Leftrightarrow 2\lfloor\frac{r}{rk(X)}\rfloor\geq d.
$$
\end{theorem}
Here $rk(X)$ is the rank of $X$ (see Definition \ref{defofrank}), $V_{\lambda_k}$ is the $G$-module with highest weight $\lambda_k$ and $V_{\lambda_k}$ is not self-dual if and only if $X$ is one of the following: Grassmannian $Gr(a,a+b)(a\neq b)$, Spinor variety $\mathbb{S}_n$($n$ is odd), Cayley plane $\mathbb{OP}^2$. For the classification of IHSS, see the table in Corollary \ref{corpt}.

Note that Fu and Liu extend Theorem \ref{maincoh} to the case that $X$ is a rational homogeneous space of Picard number 1 in \cite[Theorem 1.14]{FL21} with completely different techniques. 

The proof of Theorem \ref{maincoh} is based on the irreducible decomposition of $Sym^r T_X$. Let $X=G/P_k$ be an IHSS. In Section \ref{sectionMFA}, we show that the action of the reductive part of $P_k$ on the tangent bundle $T_X$ is multiplicity free, skew multiplicity free and weight multiplicity free. The theory of multiplicity free actions helps us to study the irreducible decomposition of $Sym^r T_X$ systematically and we can show that the irreducible decomposition of $Sym^r T_X$ has the form (see Remark \ref{irrdecom}):
$$
Sym^r T_X=\bigoplus_{n_1i_1+n_2i_2+...+n_ti_t=r\atop i_1,i_2,...,i_t\geq 0}E_{i_1\omega_1+...+i_t\omega_t}.
$$
where $\omega_1,...,\omega_t$ are the fundamental highest weights of $T_X^*$, $n_i$ is the degree of $\omega_i$, $t$ is the rank of $X$ and $E_{i_1\omega_1+...+i_t\omega_t}$ is the irreducible homogeneous vector bundle with highest weight $i_1\omega_1+...+i_t\omega_t$.
In Section \ref{IDIHSS}, we determine each $\omega_i$ and $n_i$ and formulate the irreducible decomposition of $Sym^r T_X$ explicitly.

As an immediate corollary of Theorem \ref{maincoh}, we obtain the pseudoeffective threshold of $T_X$.

\begin{corollary}\label{corpt}
Let $X=G/P_k$ be an irreducible Hermitian symmetric space of compact type. The pseudoeffective threshold of $T_X$ is given by
$$
\Delta(X,\mathcal{O}_X(1))=
\left\{
\begin{aligned}
 &\frac{1}{rk(X)}\ \ \ \ \text{if $V_{\lambda_k}$ is not self-dual;}\\
 &\frac{2}{rk(X)}\ \ \ \ \text{if $V_{\lambda_k}$ is self-dual.}
\end{aligned}
\right.
$$
Hence the pseudoeffective threshold $\Delta=\Delta(X,\mathcal{O}_X(1))$ of each IHSS is given as follows:
\begin{table}[h]
\renewcommand\arraystretch{1.2}
\begin{tabular}{|c|c|c|c|c|c|c|}
\hline
$\mathbf{IHSS}$  & $Gr(a,a+b)$ &  \begin{tabular}[c]{@{}c@{}}$\mathbb{Q}^n$\\$(n\geq 3)$\end{tabular}  &  \begin{tabular}[c]{@{}c@{}}$Lag(n,2n)$\\$(n\geq 2)$\end{tabular}  &  \begin{tabular}[c]{@{}c@{}}$\mathbb{S}_n$\\$(n\geq 4)$\end{tabular}   & $\mathbb{OP}^2$   &  $E_7/P_7$ \\ \hline
 $\Delta$   &\begin{tabular}[c]{@{}c@{}}$\frac{1}{min\{a,b\}}$ if $a\neq b$; \\
$\frac{2}{a}$ if $a=b$\end{tabular}           &  $1$                        & $\frac{2}{n}$              &\begin{tabular}[c]{@{}c@{}}$\frac{2}{n-1}$ if $n$ is odd;\\
 $\frac{4}{n}$   if $n$ is even\end{tabular}  &  $\frac{1}{2}$  &  $\frac{2}{3}$ \\ \hline
\end{tabular}
\end{table}
\end{corollary}

If we let $d=1$ in the Theorem \ref{maincoh}, we get the following cohomological characterization of the rank of a Hermitian symmetric space of compact type.
\begin{corollary}\label{corrank}
Let $X=G/P$ be a Hermitian symmetric space of compact type. Then
$$
H^0(X,Sym^r T_X \otimes \mathcal{O}_X(-1))\neq 0 \Leftrightarrow r\geq rk(X).
$$
\end{corollary}

The properties of the fundamental highest weights of $T_X^*$ and their degrees (see Proposition \ref{fhw}) help us to recover Theorem \ref{maincoh} and establish the following.
\begin{theorem}\label{minimal embedding}
Let $X=G/P$ be a Hermitian symmetric space of compact type.
\begin{enumerate}[$(1)$]
  \item If $r=rk(X)$, there is a $G$-module isomorphism:
  $$
  H^0(X,Sym^r T_X \otimes \mathcal{O}_X(-1))\simeq H^0(X,\mathcal{O}_X (1))^*;
  $$
  \item We have the following vanishing result:
$$
H^p(X,Sym^r T_X\otimes \mathcal{O}_X(-1))=0 \text{ for all $p\geq 1$ and $r\geq 1$}.
$$
\end{enumerate}
\end{theorem}
It is interesting that Thereom \ref{minimal embedding}(1) implies that there is a closed embedding
$$
X\hookrightarrow\mathbb{P}(H^0(X,Sym^{rk(X)}T_X\otimes \mathcal{O}_X(-1))).
$$

The irreducible decomposition of $Sym^r T_X$ also allows us to calculate the cohomology group $H^p(X,Sym^rT_X\otimes \mathcal{O}_X(-d))$ by using the Borel-Weil-Bott's theorem. Notice that Snow calculated the cohomology group $H^p(X,\wedge^r\Omega_X\otimes \mathcal{O}_X(d))$ in \cite{Snow1} and \cite{Snow2}. So our results can be regarded as parallel to his for symmetric products of tangent bundles.
Moreover, we will also consider the irreducible decompositions of exterior products of $T_X$ and by Serre duality, one can recover some results in \cite{Snow1} and \cite{Snow2}.

\section{Basic notations and facts}
Let $G$ be a semisimple Lie group of rank $n$ and let $\mathfrak{g}$ be its corresponding Lie algebra. Fix a Cartan subalgebra $\mathfrak{h}$ of $\mathfrak{g}$, let $\Delta =\{\alpha_{1},\alpha_{2},...,\alpha_{n}\}$ be the set of simple roots of  $\mathfrak{g}$ with respect to $\mathfrak{h}$(here the order of simple roots is given by Bourbaki's convention). Let $\{\lambda_{1},\lambda_{2},...,\lambda_{n}\}$ be the set of corresponding fundamental dominant weights, and $\delta=\sum \lambda_i$ the sum of all fundamental dominant weights. Let $P_i$ be the maximal parabolic subgroup corresponding to the simple root $\alpha_{i}$. If $\lambda$ is a dominant weight of $\mathfrak{g}$, we denote by $V_{\lambda}$ the irreducible representation of $\mathfrak{g}$ with highest weight $\lambda$.

Let $P$ be a parabolic subgroup of $G$, which can be decomposed as $P=T_{P}U_{P}S_{P}$, where $T_{P}$ is a torus, $U_P$ is the unipotent radical of $P$, and $S_P$ is the semisimple part of $P$. Let $\mathfrak{p}=Lie(P)$, $\mathfrak{u_{\mathfrak{p}}}=Lie(U_P)$ be their Lie algebras. Given a $P$-module $V$, we have a homogeneous vector bundle $E=G\times_{P}V$. In fact the category of $P$-modules is equivalent to the category of homogeneous vector bundles on the rational homogeneous space $G/P$. For this reason, we will not distinguish $P$-modules and homogeneous vector bundles on $G/P$.

Let $\{\alpha_1',...,\alpha_k'\}$ be a subset of simple roots, and $\{\lambda_1',...,\lambda_k'\}$ the corresponding dominant weights. Let $P$ be the parabolic subgroup of $G$ defined by $\{\alpha_1',...,\alpha_k'\}$, then all irreducible $P$-modules can be classified \cite[Proposition 10.9]{Ottaviani1995rational} as follows:
\begin{center}
  $V\otimes L_{\lambda_1'}^{n_1}\otimes...\otimes L_{\lambda_k'}^{n_k}$\ ,
\end{center}
where $V$ is a representation of $S_P$, $n_i\in \mathbb{Z}$ and $L_{\lambda_i'}$ is the one-dimensional representation of $P$ defined by the dominant weight $\lambda_i'$. If $\lambda$ is the highest weight of $V$ as an $S_P$-module, we define $\lambda+\Sigma n_i \lambda_i'$ to be the highest weight of the $P$-module $V\otimes L_{\lambda_1'}^{n_1}\otimes...\otimes L_{\lambda_k'}^{n_k}$. We use $E_{\lambda+\Sigma n_i \lambda_i'}$ to represent the $P$-module with highest weight $\lambda+\Sigma n_i \lambda_i'$.

\begin{definition}
  A rational homogeneous space $G/P$ is called a Hermitian symmetric space of compact type(HSS for short) if the adjoint representation $ad:\mathfrak{p}\rightarrow End(\mathfrak{g}/\mathfrak{p})$ is trivial on $\mathfrak{u_{\mathfrak{p}}}$, or equivalently if $[\mathfrak{u_{\mathfrak{p}}},\mathfrak{u_{\mathfrak{p}}}]=0$.
\end{definition}

\begin{remark}
    A $\mathfrak{p}$-module is completely reducible if and only if this $\mathfrak{p}$-module is acted by $\mathfrak{u_{\mathfrak{p}}}$ trivially.  Hence if $X$ is an HSS, then the tangent bundle $T_X$ as well as its symmetric product $Sym^r T_X$ is completely reducible. If $T_X$ is irreducible, then we call $X$ an irreducible Hermitian symmetric space of compact type(IHSS for short).
\end{remark}

The main tool to calculate the cohomology of homogeneous vector bundles on rational homogeneous spaces is the following:
\begin{theorem}[Borel-Weil-Bott]
  Let $X=G/P$ be a rational homogeneous space and $E_{\lambda}$ an irreducible homogenous vector bundle with highest weight $\lambda$ on $X$.
  \begin{enumerate}[$(1)$]
    \item If $\lambda+\delta$ is singular, then $H^p(X,E_{\lambda})=0$ for all $p$;
    \item If $\lambda+\delta$ is regular of index $i$, then $H^p(X,E_{\lambda})=0$ for all $p\neq i$ and $H^i(X,E_{\lambda})=V_{\sigma (\lambda+\delta)-\delta}$, where $\sigma$ is the unique element in the Weyl group of $G$ such that $\sigma(\lambda+\delta)$ is in the fundamental Weyl chamber of $G$.
  \end{enumerate}
\end{theorem}
Recall that $\lambda$ is singular if there exists a positive root $\alpha$ of $G$ such that $(\lambda,\alpha)=0$, where $(,)$ is the Killing form, and $\lambda$ is regular of index $i$ if it is not singular and the number of positive roots $\alpha$ of $G$ such that $(\lambda,\alpha)<0$ is $i$.

Now we recall the definition of Schur functors and we refer the reader to \cite[Lecture 6]{fulton2013representation} for more details.

Let $\mu=(\mu_1,\mu_2,...,\mu_m)$ be a partition of $n$, denoted by $\mu\vdash n$, where $\mu_1\geq \mu_2\geq ...\geq\mu_m\geq 0$. We can associate $\mu$ with a Young diagram with $\mu_i$ boxes in the $i$-th row and we say such a Young diagram is of shape $\mu$. For example, the partition $\mu=(4,3,1)$ corresponds to the following Young diagram:
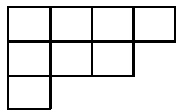
\begin{table}[h]
\begin{tabular}{|p{2mm}|p{2mm}p{2mm}p{2mm}}
\hline
 & \multicolumn{1}{l|}{} & \multicolumn{1}{l|}{} & \multicolumn{1}{l|}{} \\ \hline
 & \multicolumn{1}{l|}{} & \multicolumn{1}{l|}{} &                       \\ \cline{1-3}
 &                       &                       &                       \\ \cline{1-1}
\end{tabular}
\end{table}

The \emph{conjugation} of a partition $\mu$ is the partition $\mu'$ whose Young diagram is the transpose of the Young diagram of $\mu$. Let $\mu'=(\mu_1',\mu_2',...,\mu_l')$, then $\mu_i'=\# \{j|\mu_j\geq i\}$. For example, the conjugation of $\mu=(4,3,1)$ is $\mu'=(3,2,2,1)$.

Suppose the main diagonal of the Young diagram of the partition $\mu=(\mu_1,...,\mu_m)$ consists of $s$ boxes. Let $r_i=\mu_i-i$ and $c_i=\mu_i'-i$. We denote the partition $\mu$ by
$$
\mu=(r_1,...,r_s|c_1,...,c_s).
$$
For example, the partition $\mu=(4,3,1)$ can be written as $\mu=(3,1|2,0)$.

Inscribing the integers $1,2,..,n$ into the empty cells (in any order) of the Young diagram of shape $\mu$, we obtain a Young tableau $T$. Let

\begin{equation}\nonumber
\begin{aligned}
 W_1     &=\{g\in S_n:g\text{ preserves each row of } T\},\\
 W_2     &=\{g\in S_n:g\text{ preserves each column of }T\},\\
 a_\mu &=\sum\limits_{g\in W_1}e_g,\ \ \ \ \ \  \ \ b_\mu=\sum\limits_{g\in W_2}sgn(g)e_g,
\end{aligned}
\end{equation}
where $e_g$ is the element in the group algebra $\mathbb{C}[S_n]$ corresponding to $g$.

Let $V$ be a finite dimensional complex vector space. Both $a_\mu$ and $b_\mu$ act naturally on $V^{\otimes n}$, and
\begin{equation}\nonumber
\begin{aligned}
 Im(a_\mu) &=Sym^{\mu_1}V\otimes Sym^{\mu_2}V\otimes...\otimes Sym^{\mu_m}V,\\
 Im(b_\mu) &=\bigwedge^{\mu_1'}V\otimes \bigwedge^{\mu_2'}V\otimes...\otimes \bigwedge^{\mu_l'}V.
\end{aligned}
\end{equation}
Set $c_\mu=a_\mu b_\mu\in \mathbb{C}[S_n]$, which is called a \emph{Young symmetrizer}.\\

\begin{definition}
Let $\mu$ be a partition of $n$ and $V$ a finite dimensional complex vector space. Denote the image of $c_\mu$ on $V^{\otimes n}$ by $\mathbb{S}_\mu V$. The functor $\mathbb{S}_\mu:V\mapsto \mathbb{S}_\mu V$ is called the \emph{Schur functor} corresponding to $\mu$. The \emph{plethysm} is a composition of two Schur functors.
\end{definition}
\begin{remark}
 The definition of the Schur functor only depends on the shape of the Young tableau $T$. There are two special Schur functors which are familiar to us: the symmetric power $Sym^n$, which corresponds to the partition $(n)$ and the exterior power $\bigwedge\limits^n$, which corresponds to the partition (1,...,1) with $n$ parts. Moreover, we can similarly define Schur functors on the category of vector bundles on an algebraic variety.
\end{remark}

Now we recall some basic properties of Schur functors and Schur polynomials. Fix a partition $\mu=(\mu_1,...,\mu_m)$. Let $f:V\rightarrow V$ be a diagonalizable linear map of $k-$dimensional vector space $V$ $(k\geq m)$ and $y_1,y_2,...,y_k$ the eigenvalues of $f$. Then the trace of the induced linear map $\mathbb{S}_\mu f:\mathbb{S}_\mu V \rightarrow \mathbb{S}_\mu V$ is exactly $s_\mu (y_1,y_2,...,y_k)$, where $s_{\mu}(x_1,...,x_k)$ is the Schur polynomial defined as follows:
$$
s_{\mu}(x_1,...,x_k)=\frac{det((x_j^{\mu_i+k-i})_{i,j})}{det((x_j^{k-i})_{i,j})}.
$$
The Schur polynomial can be written as (\cite[Equation (A.19)]{fulton2013representation})
\begin{equation}\label{Schurpoly}
  s_\mu (x_1,...,x_k)=M_\mu+\sum_{\gamma < \mu} K_{\mu \gamma}M_\gamma,
\end{equation}
where
\begin{equation}\nonumber
  M_{\gamma}=\sum_{\beta}X^{\beta},
\end{equation}
summed over all distinct permutations $\beta=(\beta_1,...,\beta_k)$ of the partition $\gamma=(\gamma_1,...,\gamma_k)$ with $X^\beta=x_{1}^{\beta_1}x_{2}^{\beta_2}...x_{k}^{\beta_k}$ and $K_{\mu\gamma}$ are Kostka numbers. There is a one-to-one correspondence between the monomials in the displayed equation (\ref{Schurpoly}) of $s_\mu(x_1,x_2,...,x_k)$ and 1-dimensional eigenspace of $\mathbb{S}_\mu f$.

\begin{proposition}[{\cite[P160]{macdonald1998symmetric}}]
\emph{Let $\omega$ and $\mu$ be two partitions. Then there is a decomposition:
\begin{center}
  $\mathbb{S}_\mu \circ \mathbb{S}_\omega =\bigoplus\limits_\nu \mathbb{S}_\nu^{\oplus a_{\mu\omega}^{\nu}}$
\end{center}
for some non-negative integer $a_{\mu\omega}^{\nu}$}.
\end{proposition}

The plethysm problem is to determine the nonnegative integer $a_{\mu\omega}^{\nu}$. However, it is still an open problem to find a combinatorial formula for the coefficients $a_{\mu\omega}^{\nu}$ and there are only partial solutions in some special cases. We refer readers to \cite[I.8]{macdonald1998symmetric} for more details.

Here are four particular examples, which will be used in this paper.
\begin{example}[{\cite[I.8.Example 6]{macdonald1998symmetric}}]\label{plethysmexample}
\hspace*{\fill}
\begin{enumerate}[$(a)$]
  \item $Sym^r\circ Sym^2=\bigoplus\limits_{\mu \vdash 2r}\mathbb{S}_{\mu}$, summed over all the even partitions $\mu$ of $2r$(i.e. all the parts of $\mu$ are even.)\\
  \item $Sym^r\circ \bigwedge\limits^2=\bigoplus\limits_{\mu \vdash 2r}\mathbb{S}_{\mu'}$, summed over all the even partitions $\mu$ of $2r$. Here $\mu'$ represents the conjugation of the partition $\mu$.\\
  \item $\bigwedge\limits^r\circ \bigwedge\limits^2=\bigoplus\limits_{\mu}\mathbb{S}_{\mu}$, summed over all the partitions $\mu$ of the form $(c_1-1,...,c_m-1|c_1,...,c_m)$, where $c_1>...>c_m>0$ and $c_1+...+c_m=r$.
  \item $\bigwedge\limits^r\circ Sym^2=\bigoplus\limits_{\mu}\mathbb{S}_{\mu'}$, summed over all the partitions $\mu$ of the form $(c_1-1,...,c_m-1|c_1,...,c_m)$, where $c_1>...>c_m>0$ and $c_1+...+c_m=r$.
\end{enumerate}
\end{example}

\section{pseudoeffective thresholds of tangent bundles}\label{PT}
Let $X$ be a smooth projective variety and $A$ an ample divisor on $X$. The pseudoeffective threshold $\Delta(X,A)$ defined in Definition \ref{bignessdefect} is a number measuring the bigness of the tangent bundle of $X$ with respect to the ample divisor $A$. In this section, we will give some properties of $\Delta(X,A)$ and give a characterization of projective spaces and hyperquadrics in terms of $\Delta(X,A)$(see Theorem \ref{picardnumber1}).

\begin{definition}
    Let $X$ be a smooth projective variey of dimension $n$ and $A$ an ample divisor on $X$. Let $\mathscr{E}$ be a torsion free sheaf on $X$. We define the slope of $\mathscr{E}$ with respect to $A$ to be $\mu_A (\mathscr{E})=\frac{c_1(\mathscr{E})\cdot A^{n-1}}{rk(\mathscr{E})}$.
\end{definition}

\begin{proposition}\label{cohomology}
Let $X$ be a smooth projective variety and $A$ an ample divisor on $X$. Then
$$
  \Delta(X,A)=sup\{\frac{d}{r}|H^0(X,Sym^r T_X \otimes \mathcal{O}_X(-dA))\neq 0\}.
$$
\end{proposition}

\begin{proof}
  Let $\Delta '(X,A)$ denote the right hand side of the above equation. Since
\begin{center}
  $H^0(X,Sym^r T_X \otimes \mathcal{O}_X(-dA))=H^0(\mathbb{P}(T_X),\mathcal{O}_{\mathbb{P}(T_X)}(r)\otimes \pi^{*}\mathcal{O}_X(-dA))$,
\end{center}
we see that $\Delta '(X,A)\leq \Delta(X,A)$.

For every positive integer $n$, take a rational number $\frac{p}{q}$ such that $\frac{p}{q}\geq \Delta(X,A)-\frac{1}{n}$ and $\xi-\frac{p}{q} \pi^*A$ is effective. It follows that
\begin{center}
  $H^0(\mathbb{P}(T_X),\mathcal{O}_{\mathbb{P}(T_X)}(qk)\otimes \pi^{*}\mathcal{O}_X(-pkA))\neq 0$
\end{center}
for some integer $k\in \mathbb{N}^+$. Thus
\begin{center}
  $H^0(X,Sym^{qk} T_X \otimes \mathcal{O}_X(-pkA))\neq 0$
\end{center}
and
\begin{center}
  $\Delta(X,A)-\frac{1}{n}\leq \frac{p}{q}\leq \Delta'(X,A)$
\end{center}
for every positive integer $n$. It follows that $\Delta(X,A)\leq \Delta'(X,A)$ and thus $\Delta(X,A)= \Delta'(X,A)$.
\end{proof}

Let $E$ be a vector bundle on $X$. By \cite[Proposition 2.1]{hsiao2015a} and \cite[Example 6.1.23]{lazarsfeld2004positivity}, $E$ is big if and only if
$$
H^0(X,Sym^r E\otimes L^{-1})\neq 0
$$
for some ample line bundle $L$ on $X$ and some positive integer $r$. As an immediate corollary, we have the following:
\begin{proposition}
Let $X$ be a smooth projective variety and $A$ an ample divisor on $X$. Then $T_X$ is big if and only if $\Delta(X,A)>0$.
\end{proposition}

Recall that a divisor $D$ on a smooth projective variety $X$ is pseudoeffective if it is in the closure of the effective cone of $X$. A vector bundle $E$ is called pseudoeffective if the tautological class $c_1(\mathcal{O}_{\mathbb{P}(E)}(1))$ is pseudoeffective on $\mathbb{P}(E)$. Using the pseudoeffective threshold, we give the following characterization of pseudoeffectivity of tangent bundles by \cite[Lemma 2.7]{SD18}.
\begin{proposition}
Let $X$ be a smooth projective variety and $A$ an ample divisor on $X$. Then $T_X$ is pseudoeffective if and only if $\Delta(X,A)\geq 0$.
\end{proposition}

\begin{example}\label{egbignessdefect1}
\begin{equation}\nonumber
\begin{aligned}
&\Delta(\mathbb{P}^1, \mathcal{O}_{\mathbb{P}^1}(1))=2;\\
&\Delta(\mathbb{P}^n, \mathcal{O}_{\mathbb{P}^n}(k))=\frac{1}{k} \ \ \text{if} \ n\geq 2\ and\ k\geq 1.
\end{aligned}
\end{equation}
In fact, notice that $Sym^r T_{\mathbb{P}^n}$ is an irreducible homogeneous vector bundle for every positive integer $r$. Hence one can easily get that 
$$
H^0(X,Sym^r T_{\mathbb{P}^n} \otimes \mathcal{O}_{\mathbb{P}^n}(-d))\neq 0 \Leftrightarrow r\geq d.
$$
for $n\geq 2$ by the Borel-Weil-Bott's theorem. Furthermore, one can also use \cite[Theorem 1]{symofPn} to obtain the same nonvanishing result. 
\end{example}

\begin{example}\label{egbignessdefect2}
  $$
  \Delta(Q^n, \mathcal{O}_{Q^n}(1))=1.
  $$
  where $Q^n$ is an $n$-dimensional smooth hyperquadric and $(Q^1,\mathcal{O}_{Q^1}(1))=(\mathbb{P}^1,\mathcal{O}_{\mathbb{P}^1}(2))$. By \cite[Lemma 2]{druel2013characterizations}, $\Delta(Q^n,\mathcal{O}_{Q^n}(1))\leq 1$. On the other hand, using the isomorphism $T_{Q^n}(-1)\cong \Omega_{Q^n}(1)$ and \cite[Theorem B]{bogomolov2008symmetric}, there exists a natural number $r$ such that
\begin{center}
  $H^0(X,Sym^r T_{Q^n} \otimes \mathcal{O}_{Q^n}(-r))\neq 0$.
\end{center}
Hence $\Delta(Q^n, \mathcal{O}_{Q^n}(1))=1$.
\end{example}

\begin{proposition}\label{bound}
Let $X$ be a smooth projective variety and $A$ an ample divisor on $X$. If $(X,\mathcal{O}_X(A))\neq (\mathbb{P}^1,\mathcal{O}_{\mathbb{P}^1}(1))$, then $\Delta(X,A)\leq 1$.
\end{proposition}
\begin{proof}
  If conversely $\Delta(X,A)>1$, then there exists a positive integer $r$ such that
$$
H^0(X,Sym^r T_X\otimes \mathcal{O}_X(-rA))\neq 0.
$$
By \cite[Theorem B]{druel2013characterizations}, $X\simeq \mathbb{P}^n$ or $Q^n$. If $X\simeq \mathbb{P}^n$, by Example \ref{egbignessdefect1}, $\mathcal{O}_{\mathbb{P}^n}(A)=\mathcal{O}_{\mathbb{P}^n}(1)$ and $\Delta(\mathbb{P}^n,\mathcal{O}_{\mathbb{P}^n}(1))=1$. If $X\simeq Q^n$, then by Example \ref{egbignessdefect2}, $\mathcal{O}_{Q^n}(A)=\mathcal{O}_{Q^n}(1)$ and $\Delta(Q^n,\mathcal{O}_{Q^n}(1))=1$.
\end{proof}

Example \ref{egbignessdefect1}, Example \ref{egbignessdefect2} and Proposition \ref{bound} prompt us to propose Conjecture \ref{conj}. We will prove this conjecture under the assumption that the Picard number of $X$ is 1 after some preparations.

We use the terminology of \cite{kollar1996rational} for families of rational curves. For a uniruled projective variety $X$, let $\mathcal{K}$ be an irreducible component of $RatCurve^n(X)$. We say $\mathcal{K}$ is a dominating family of rational curves or a covering family of rational curves if the corresponding universal family dominates $X$. A covering family $\mathcal{K}$ is called minimal if, for general point $x\in X$, the subfamily of $\mathcal{K}$ parametrizing curves through $x$ is proper.

Fix a minimal dominating family $\mathcal{K}$ of rational curves on $X$. Let $x\in X$ be a general point. Denote by $\mathcal{K}_x$ the normalization of the subscheme of $\mathcal{K}$ parametrizing rational curves passing through $x$. Define the tangent map $\tau_x: \mathcal{K}_x \dashrightarrow \mathbb{P}(T_x^* X)$ by sending a curve that is smooth at $x$ to its tangent direction at $x$. We denote by $\mathcal{C}_x$ the closure of the image of $\tau_x$ in $\mathbb{P}(T_x^* X)$. Let $\overline{\mathcal{K}}$ be the closure of $\mathcal{K}$ in $Chow(X)$. We say $x,y\in X$ are $\mathcal{K}$-equivalent if they can be connected by a chain of 1-cycles in $\overline{\mathcal{K}}$. By \cite[\uppercase\expandafter{\romannumeral4}.4.16]{kollar1996rational}, there is a proper surjective morphsim $\pi_0:X_0\rightarrow Y_0$ from a dense open subset of $X$ onto a normal variety whose fibers are $\mathcal{K}$-equivalence class. We call this map the $\mathcal{K}$-rationally connected quotient of $X$.

Let $\ell$ be the rational curve corresponding to a general point in $\mathcal{K}$ with normalization morphism $f:\mathbb{P}^1\rightarrow \ell \subset X$. We use $[f]$ or $[\ell]$ to denote the point in $\mathcal{K}$ corresponding to $\ell$. For general $[f]\in \mathcal{K}$, let $f^*T_X^+$ be the positive part of $f^*T_X$, i.e.
$$
f^*T_X^+=im[H^0(\mathbb{P}^1,f^*T_X(-1))\otimes \mathcal{O}_{\mathbb{P}^1}(1)\rightarrow f^*T_X].
$$
By \cite[Proposition 2.3]{hwang2001geometry}, $\mathbb{P}((f^*T_X^+)_x^*)$ is the projective tangent space of $\mathcal{C}_x$ at $\tau_x([f])$ in $\mathbb{P}(T_x^* X)$ if $[f]\in \mathcal{K}_x$ is general.

The proof of the following proposition is similar to that of \cite[Theorem 6.3]{araujo2008cohomological}.

\begin{proposition}\label{picardnumber1}
  Let $X$ be an $n-$dimensional smooth projective variety with $\rho(X)=1$ and let $A$ be an ample divisor on $X$. Suppose that $\Delta(X,A)\geq 1$. Then $(X,A)\simeq (\mathbb{P}^n,\mathcal{O}_{\mathbb{P}^n}(1))$or $(Q^n,\mathcal{O}_{Q^n}(1))(n\neq 2)$.
\end{proposition}
\begin{proof}
  By Proposition \ref{cohomology}, there exist positive integers $r$ and $d$ satisfying $max\{\frac{3}{4},1-\frac{1}{n}\}<\frac{d}{r}\leq 1$ and
\begin{center}
  $H^0(X,Sym^r T_X\otimes \mathcal{O}_X(-dA))\neq 0$.
\end{center}
By \cite[Corollary 8.6]{miyaoka1985deformations} or \cite[Theorem 0.1]{campana2011geometric}, $X$ is uniruled, hence X is Fano as $\rho(X)=1$. Let $\mathcal{K}$ be a minimal covering family of rational curves on $X$. Let $\mathscr{E}$ be the maximal destabilizing subsheaf of $T_X$. By \cite[Lemma 6.2]{araujo2008cohomological}, the slope $\mu_A(\mathscr{E})$ of $\mathscr{E}$ satisfies
\begin{equation}\nonumber
    \mu_A(\mathscr{E})\geq \frac{\mu_A(dA)}{r}=\frac{d}{r}\mu_A(A).
\end{equation}
Hence
\begin{equation}\nonumber
  \frac{deg\ f^*\mathscr{E}}{rk\ \mathscr{E}}\geq \frac{d}{r}deg\ f^*A
\end{equation}
for general elements $[f]\in \mathcal{K}$ as $\rho (X)=1$. By \cite[Chapter II, Proposition 3.7]{kollar1996rational}, $f^*\mathscr{E}$ is locally free  for general $f$. If $f^*\mathscr{E}$ is ample, observe that the $\mathcal{K}$-rationally connected quotient of $X$ is trivial since $\rho (X)=1$, hence $X$ is a projective space by \cite[Proposition 2.7]{araujo2008cohomological}. If $rk \mathscr{E}=1$, then $f^*\mathscr{E}$ is ample. Thus we may assume that $rk\ \mathscr{E}\geq 2$. Then $f^*\mathscr{E}$ is a subsheaf of $f^*T_X\simeq \mathcal{O}_{\mathbb{P}^1}(2)\oplus \mathcal{O}_{\mathbb{P}^1}(1)^{\oplus m}\oplus \mathcal{O}_{\mathbb{P}^1}^{\oplus (n-m-1)}$ and $\frac{deg\ f^*\mathscr{E}}{rk\ \mathscr{E}}\leq \frac{3}{2}$. If $deg\ f^*A\geq 2$, then
$$
\frac{deg\ f^*\mathscr{E}}{rk\ \mathscr{E}}\geq \frac{d}{r}deg\ f^*A>\frac{3}{4}\cdot 2=\frac{3}{2},
$$
which is a contradiction. Hence $deg\ f^*A=1$ and either $f^*\mathscr{E}$ is ample or $f^*\mathscr{E}\simeq \mathcal{O}_{\mathbb{P}^1}(2)\oplus \mathcal{O}_{\mathbb{P}^1}(1)^{\oplus (rk(\mathscr{E})-2)}\oplus \mathcal{O}_{\mathbb{P}^1}$ since $\frac{d}{r}>1-\frac{1}{n}$.

 If $f^*\mathscr{E}$ is not ample, then $\mathcal{O}_{\mathbb{P}^1}(2)\subset f^*\mathscr{E}$ for general $[f]\in \mathcal{K}$. Thus by \cite[Proposition 2.3]{hwang2001geometry}, $(f^*T_X^+)_y \subset (f^*\mathscr{E})_y$ for general $y\in \mathbb{P}^1$ and general $[f]\in \mathcal{K}$. Since $f^*\mathscr{E}$ is a subbundle of $f^*T_X$, we have an inclusion of sheaves $f^*T_X^+\hookrightarrow f^*\mathscr{E}$ and hence $det(f^*\mathscr{E})=f^*\omega_X^{-1}$. It follows that $det(\mathscr{E^{**}})=\omega_X^{-1}$ since $\rho (X)=1$, and thus $0\neq h^0(X,\wedge^{rk\mathscr{E}}T_X\otimes \omega_X)=h^{n-rk\mathscr{E}}(X,\mathcal{O}_X)$. Note that $X$ is Fano, hence $h^{n-rk\mathscr{E}}(X,\mathcal{O}_X)=0$ unless $n=rk(\mathscr{E})$. If $rk(\mathscr{E})=n$, then we have $\omega_X^{-1}\simeq \mathcal{O}_X(A)^{\otimes n}$. Therefore, $X$ is a hyperquadric by \cite{kobayashi1973characterizations}.
\end{proof}

\begin{remark}
 It is pointed out by St\'{e}phane Druel, Andreas Höring and Jie Liu that the Proposition \ref{picardnumber1} can be derived by \cite[Theorem 1.1]{MR3331171}. In fact, let $\mu_{A}^{max}(T_X)$ be the slope of the maximal destabilizing subsheaf of $T_X$ with respect to $A$. By \cite[Lemma 6.2]{araujo2008cohomological},  $\mu_{A}^{max}(T_X)\geq\mu_A(A)$. On the other hand, by \cite[Theorem 1.1]{MR3331171}, if $X$ is not a projective space, then $\mu_{A}^{max}(T_X)\leq\mu_A(A)$. Hence $\mu_{A}^{max}(T_X)=\mu_A(A)$. Then the result follows from \cite{kobayashi1973characterizations}.
\end{remark}

\section{multiplicity free actions}\label{sectionMFA}
Let $R$ be a linear reductive connected algebraic group over the complex number field and $V$ an $R$-module. Let $\mathcal{P}(V)=Sym(V^*):=\bigoplus\limits_{i\geq 0}Sym^i(V^*)$ be the polynomial algebra over $V$ and the action of $R$ on $V$ can be extended to $\mathcal{P}(V)$ naturally. So $\mathcal{P}(V)$ can be decomposed into a direct sum of irreducible $R$-modules. We say $V$ is multiplicity free if each irreducible component of $\mathcal{P}(V)$ has multiplicity one.

\begin{proposition}[{\cite[Theorem 6.2]{servedio1973prehomogeneous} and \cite[Theorem 1.2]{leahy1998classification}}]\label{equivalent-condition-of-MF}
The following statements are equivalent:
\begin{enumerate}[$(1)$]
  \item $V$ is a multiplicity free $R$-module;
  \item $V$ is a spherical $R$-variety, i.e. there is a Borel subgroup $B$ of $R$ which acts on $V$ with an open orbit;
  \item $B$ has only finitely many orbits in $V$.
\end{enumerate}
\end{proposition}

Let $B$ be a Borel subgroup of $R$, i,e. a maximal connected solvable algebraic group of $R$. Then $B$ can be written as the semidirect product $B=TN$, where $T$ is a maximal torus in $R$ and $N$ is a maximal unipotent subgroup of $R$. Let $\mathfrak{t}$ be the Lie algebra of $T$.

Let $h_{\lambda}\in \mathcal{P}(V)$ be a highest weight vector with highest weight $\lambda$. Regard $h_{\lambda}$ as a polynomial and suppose the irreducible decomposition of $h_{\lambda}$ is $h_{\lambda}=p_1^{m_1}...p_k^{m_k}$. Then each irreducible factor $p_j$ is a highest weight vector by \cite[Lemma 3.1.1]{benson2004multiplicity} and the irreducible decomposition of the polynomial algebra of $V$ is given by the following:

\begin{proposition}[{\cite[Proposition 3.3.1]{benson2004multiplicity}}]\label{pvd}
Let $\Lambda$ be the subset of $\mathfrak{t}^*$ consisting of all the highest weights of irreducible representations occurring in $\mathcal{P}(V)$. Let $\Lambda'=\{\lambda\in \Lambda:h_{\lambda}$ is irreducible$\}$. Then $\Lambda'$ is a finite $\mathbb{Q}$-linear independent subset of $\mathfrak{t}^*$ and $\Lambda$ is an additive semigroup freely generated by $\Lambda'$. Suppose that $\Lambda'=\{\omega_1,...,\omega_t\}$. Then we have the following irreducible decomposition:
$$
\mathcal{P}(V)=\bigoplus_{\lambda}V_{\lambda},
$$
where the sum is taken over all $\mathbb{N}$-linear combinations $\lambda=m_1\omega_1+...+m_t\omega_t$. The highest weight vector corresponding to $\lambda=m_1\omega_1+...+m_t\omega_t$ is $h_{\lambda}=h_{\omega_1}^{m_1}...h_{\omega_t}^{m_t}$.

\end{proposition}

\begin{definition}
In the setting of Proposition \ref{pvd}, we call $\omega_1,...,\omega_t$ the fundamental highest weights and $h_{\omega_1},...,h_{\omega_t}$ the fundamental highest weight vectors. The cardinality $t$ of the set $\Lambda'$ is called the $rank$ of this multiplicity free action. The degree of $\omega_i(1\leq i\leq t)$ is the $degree$ of $h_{\omega_i}$.
\end{definition}

Irreducible multiplicity free actions have been classified by Kac in \cite[Theorem 3]{kac1980some}. By Kac's classification, we can show the following:

\begin{proposition}\label{IHSS-MF}
Let $X=G/P$ be an IHSS and let $R$ be the reductive part of $P$. Then the action of $R$ on $T_X$ is multiplicity free.
\end{proposition}

If $X=G/P$ is an IHSS, it is well-known that the action of $P$ on the tangent bundle has finitely many orbits and thus Proposition \ref{IHSS-MF} can also be deduced from Proposition \ref{equivalent-condition-of-MF}.

\begin{definition}\label{defofrank}
Let $X=G/P$ be an IHSS.  The action of the reductive part of $P$ on the tangent bundle $T_X$ of $X$ is multiplicity free.  We define the $rank$ of $X$ to be the rank of this action, denoted by $rk(X)$. The $rank$ of an HSS is defined to be the sum of the ranks of all its irreducible components. The following is the list of ranks of IHSS (see \cite[Section 6.1,Table 3]{benson2004multiplicity}):
\end{definition}

\begin{table}[h]
 \renewcommand\arraystretch{1.1}
 \begin{tabular}{|c|c|c|c|c|c|c|}
 \hline
 $\mathbf{IHSS}$  & $Gr(a,a+b)$ &  \begin{tabular}[c]{@{}c@{}}$\mathbb{Q}^n$\\$(n\geq 3)$\end{tabular}  &  \begin{tabular}[c]{@{}c@{}}$Lag(n,2n)$\\$(n\geq 2)$\end{tabular}  &  \begin{tabular}[c]{@{}c@{}}$\mathbb{S}_n$\\$(n\geq 4)$\end{tabular}   & $\mathbb{OP}^2$   &  $E_7/P_7$ \\ \hline
 $\mathbf{rank}$  & $min\{a,b\}$           &  2                        &  $n$             &  $\lfloor\frac{n}{2}\rfloor$  &  2  &  3  \\ \hline
 \end{tabular}
\end{table}

\begin{remark}
There are some geometric descriptions of the rank of an IHSS, e.g.
  it is the number of the orbits of the action of the isotropy subgroup $P$ at a base point $x\in X$ on the tangent space $\mathbb{P}(T_x X)$ and it is also the length of this IHSS. For more detail, see \cite[Section 6.2]{hwang2005geometry}. There is also a definition of the rank of HSS in terms of differential geometry (see \cite[Chapter V, Section 6]{helgason1979differential}).
\end{remark}

\begin{remark}\label{irrdecom}
Let $X$ be an IHSS and let $\omega_1,...,\omega_t$ be the fundamental highest weights of $T_X^*$. Assume that the degree of $\omega_i$ is $n_i$. Proposition \ref{pvd} implies that the irreducible decomposition of $Sym^r T_X$ is
$$
Sym^r T_X=\bigoplus_{n_1i_1+n_2i_2+...+n_ti_t=r\atop i_1,i_2,...,i_t\geq 0}E_{i_1\omega_1+...+i_t\omega_t}.
$$

 In the next section, we will give all the fundamental highest weights and their degrees for the dual of tangent bundles of IHSS, which will enable us to use the Borel-Weil-Bott's theorem to compute their cohomologies.
\end{remark}

\begin{remark}
Let $R$ be a linear reductive connected algebraic group over the complex number field and $V$ an $R$-module. We say that $V$ is skew multiplicity free if each irreducible component of the exterior algebra $\bigwedge V=\bigoplus\limits_{i\geq 0}\bigwedge\limits^i V$ has multiplicity one. Irreducible skew multiplicity free actions also have been classified up to an equivalence by Pecher in \cite[Theorem 10]{SMF}. Unfortunately, if $V$ is skew multiplicity free, the set of highest weights of irreducible components of the exterior algebra $\bigwedge V$ has no semigroup structure and I don't know whether there is a geometric description of skew multiplicity free actions.

Let $X=G/P$ be an IHSS and and let $R$ be the reductive part of $P$. By the classification of irreducible skew multiplicity free actions, the action of $R$ on the tangent bundle $T_X$ is skew multiplicity free, which also implies that $T_X$ is weight multiplicity free, i.e. each weight of $T_X$ has multiplicity one by \cite[Proposition 4.5.3]{MFA}.
\end{remark}

\section{irreducible decompositions of symmetric products of tangent bundles of IHSS}\label{IDIHSS}
In this section, we calculate the irreducible decomposition of $Sym^r T_X$ for each IHSS $X=G/P_k$. First we establish an important lemma, which reduces the problem of representations of reductive groups to the semisimple case.  Let $U_S$ be a maximal unipotent subgroup of $S$, where $S$ is the semisimple part of $P_k$, then the ring $\mathcal{P}(\mathfrak{g}/\mathfrak{p}_k)^{U_S}$ of the invariant functions is  generated by finitely many algebraically independent elements by \cite[Table in Page 13]{MR698847}. Hence by using \cite[Table in Page 13]{MR698847}, we can obtain the irreducible decomposition of $Sym^r( \mathfrak{g}/\mathfrak{p}_k)$ as $S$-module and thus the irreducible decomposition of $Sym^r T_X$ can be determined.
But in the case of Grassmannian, Lagrangian Grassmannian and Spinor variety, we use a more direct approach to get the irreducible decomposition of $Sym^r T_X$.

\begin{lemma}\label{linear-relation}
Let $X=G/P_k$ be an IHSS. Then there exist positive rational numbers $q_1,...,q_n$ such that for each weight $\lambda=\sum\limits_{i=1}^{n}a_i\lambda_i$ of $T_X^{\otimes r}$, 
$$
q_1 a_1+q_2 a_2+...+q_n a_n=r.
$$ 
\end{lemma}
\begin{proof}
It is sufficient to check the case that $r=1$. Let $\lambda=\sum\limits_{i=1}^{n}a_i\lambda_i$ be a weight of $T_X$. Then $\lambda$ is a positive root of $G$ with coefficient of $\alpha_k$ positive. Since $X$ is a cominuscule homogeneous variety (See \cite[Sect.2.1]{MR2421317}), the coefficient of $\alpha_k$ in the longest root of $G$ is 1. Hence if we write $\lambda$ as a linear combination of simple roots, the coefficient of $\alpha_k$ in $\lambda$ is exactly 1. Thus
$$
(\alpha_k,\lambda_k)=(\lambda,\lambda_k)=\sum_{i=1}^{n}a_i(\lambda_i,\lambda_k)
$$
Let $\lambda_i=\sum\limits_{j=1}^{n}t_{ij}\alpha_j$. Then 
$$
q_i:=\frac{(\lambda_i,\lambda_k)}{(\alpha_k,\lambda_k)}=t_{ik}
$$
and $\sum\limits_{i=1}^{n}q_i a_i$=1.
By \cite[Sect.\uppercase\expandafter{\romannumeral3}.13, Exercise 8]{MR499562}, each $q_i$ is a positive rational number.
\end{proof}

The value of $q_i$ in the Lemma \ref{linear-relation} can be found in \cite[Plate \uppercase\expandafter{\romannumeral1}-\uppercase\expandafter{\romannumeral6}]{MR1890629}. 

\subsection{Grassmannian}\label{subsection5A}
Let $X=SL(n+1)/P_k \cong Gr(k,n+1)$ with $k\leq n+1-k$. The semisimple part of $P_k$ is $S=SL(k)\times SL(n+1-k)$. The tangent bundle of $X$ can be written as $T_X=E_{\lambda_1}\otimes E_{\lambda_n}$, where $E_{\lambda_1}$ is the dual of the universal bundle with highest weight $\lambda_1$ and $E_{\lambda_n}$ is the quotient bundle $\mathcal{O}_{X}^{\oplus (n+1)}/E_{\lambda_1}^{*}$ with highest weight $\lambda_n$. The highest weight of $T_X$ is $\lambda_1+\lambda_n$.

We have the following decomposition by \cite[Exercise 6.11]{fulton2013representation}:
\begin{equation}\label{decomposition of tensors}\nonumber
\begin{aligned}
  & Sym^r T_X=Sym^r(E_{\lambda_1}\otimes E_{\lambda_n})=\bigoplus_{\mu \vdash r}(\mathbb{S}_\mu E_{\lambda_1}\otimes \mathbb{S}_\mu E_{\lambda_n}), \\
  & \bigwedge^r T_X=\bigwedge^r(E_{\lambda_1}\otimes E_{\lambda_n})=\bigoplus_{\mu \vdash r}(\mathbb{S}_\mu E_{\lambda_1}\otimes \mathbb{S}_{\mu'} E_{\lambda_n}).
\end{aligned}
\end{equation}

Now we need to calculate the highest weight of each direct summand. In fact, each direct summand above is an irreducible $P_k$-module which can be deduced from the following result whose proof is similar to that of \cite[Proposition 15.15]{fulton2013representation}.

\begin{proposition}\label{highestweight}
  Let $\mu=(\mu_1,\mu_2,...,\mu_k)$ be a partition. Then $\mathbb{S}_\mu E_{\lambda_1}$ and $\mathbb{S}_\mu E_{\lambda_n}$ are irreducible homogeneous vector bundles with highest weights:
  \begin{center}
    $\mu_1\lambda_1+\mu_2(\lambda_2-\lambda_1)+...+\mu_k(\lambda_k-\lambda_{k-1})$
  \end{center}
 and
 \begin{center}
    $\mu_1\lambda_n+\mu_2(\lambda_{n-1}-\lambda_n)+...+\mu_k(\lambda_{n-k+1}-\lambda_{n-k+2})$,
 \end{center}
  respectively.
\end{proposition}

\begin{proof}
  If we regard $E_{\lambda_1}$ as an $S$-module, then it is isomorphic to $V_{\lambda_1}$. We know that $\mathbb{S}_\mu V_{\lambda_1}$ is irreducible(\cite[Theorem 6.3(4)]{fulton2013representation}). Hence $\mathbb{S}_\mu E_{\lambda_1}$ is an irreducible $P_k$-module as well.

Now we determine the highest weight of $\mathbb{S}_\mu E_{\lambda_1}$. We observe that the weights of $E_{\lambda_1}$ are $\lambda_1,\lambda_2-\lambda_1,...,\lambda_k-\lambda_{k-1}$ and each weight space has dimension one. Moreover, each element of the Cartan subalgebra of $sl(n+1)$ acts diagonalizablely on $\mathbb{S}_\mu E_{\lambda_1}$. Hence there is a one-to-one correspondence between the monomials in equation (\ref{Schurpoly}) of $s_\mu(x_1,x_2,...,x_k)$ and 1-dimensional eigenspace of $\mathbb{S}_\mu f$. Thus the weights of $\mathbb{S}_\mu E_{\lambda_1}$ consist of all
\begin{equation}\nonumber
  \beta_1\lambda_1+\beta_2(\lambda_2-\lambda_1)+...+\beta_k(\lambda_k-\lambda_{k-1})
\end{equation}
each occuring as often as it does in the monomial $X^\beta=x_{1}^{\beta_1}x_{2}^{\beta_2}...x_{k}^{\beta_k}$ in the polynomial $s_\mu (x_1,x_2,...,x_k)$. Hence by the decomposition of Schur polynomial (equation (\ref{Schurpoly})), the highest weight of $\mathbb{S}_\mu E_{\lambda_1}$ is
 \begin{equation}\nonumber
    \mu_1\lambda_1+\mu_2(\lambda_2-\lambda_1)+...+\mu_k(\lambda_k-\lambda_{k-1}).
  \end{equation}
\end{proof}

\begin{remark}\label{decompositionGrass}
Proposition \ref{highestweight} implies that the highest weight of $\mathbb{S}_\mu E_{\lambda_1}\otimes \mathbb{S}_\mu E_{\lambda_n}$ is
\begin{equation}\nonumber
\begin{aligned}
&\ \ \ \ \mu_1\lambda_1+\mu_2(\lambda_2-\lambda_1)+...+\mu_k(\lambda_k-\lambda_{k-1})\\
&+\mu_1\lambda_n+\mu_2(\lambda_{n-1}-\lambda_n)+...+\mu_k(\lambda_{n-k+1}-\lambda_{n-k+2})\\
&=(\mu_1-\mu_2)(\lambda_1+\lambda_n)+...+(\mu_{k-1}-\mu_{k})(\lambda_{k-1}+\lambda_{n-k+2})+\mu_k(\lambda_k+\lambda_{n-k+1}).
\end{aligned}
\end{equation}
Let
$$
\begin{aligned}
  &i_j=\mu_j-\mu_{j+1},\ \ \ j=1...k-1, \\
  &i_k=\mu_k.
\end{aligned}
$$
Then the irreducible decomposition of $Sym^r T_X$ can be written as
$$
Sym^r T_X=\bigoplus_{i_1+2i_2+...+ki_k=r\atop i_1,...,i_k\geq 0}E_{i_1(\lambda_1+\lambda_n)+...+i_k(\lambda_k+\lambda_{n-k+1})}.
$$
Hence the fundamental highest weights of $T_X^*$ are $\lambda_j+\lambda_{n+1-j}$ with degree equal to $j\ (j=1,...,k)$.
\end{remark}
By the Borel-Weil-Bott theorem, we have the following result.
\begin{theorem}\label{Grassmannian}
Let $X=Gr(k,n+1)$. If $k<n+1-k$, then
$$H^0(X,Sym^r T_X \otimes \mathcal{O}_X(-d))\neq 0 \Leftrightarrow \lfloor\frac{r}{k}\rfloor\geq d;$$

If $k=n+1-k$, then $$H^0(X,Sym^r T_X \otimes \mathcal{O}_X(-d))\neq 0 \Leftrightarrow 2\lfloor\frac{r}{k}\rfloor\geq d.$$
\end{theorem}

\subsection{Hyperquadric}
\begin{lemma}[{\cite[Exercise 19.21]{fulton2013representation}} or {\cite[Table in page 13]{MR698847}}]\label{decomp}
  Let $\gamma_i(i=1,...,n)$ be the fundamental dominant weights of the simple Lie algebra of type $B_n$ or $D_n$. We have the following irreducible decomposition of $B_n$-module or $D_n$-module:
  $$
  Sym^r V_{\gamma_1}=\bigoplus_{i=0}^{\lfloor\frac{r}{2}\rfloor}V_{(r-2i)\gamma_1}.
  $$
\end{lemma}

By Lemma \ref{linear-relation}, we can show the following two lemmas.
\begin{lemma}\label{oddQ}
  Let $X=Q^{2n-1}\cong Spin(2n+1)/P_1\ (n\geq 2)$. Let $\lambda_i(i=1,2,...,n)$ be the fundamental dominant weights of the simple Lie algebra of type $B_n$. If $\lambda=\sum\limits_{i=1}^{n}a_i\lambda_i$ is the highest weight of an irreducible component of $T_X^{\otimes r}$. Then
  $$
  a_1=r-\sum_{i=2}^{n-1}a_i-\frac{1}{2}a_n.
  $$
\end{lemma}

\begin{lemma}\label{evenQ}
  Let $X=Q^{2n-2}\cong Spin(2n)/P_1(n\geq 3)$. Let $\lambda_i(i=1,2,...,n)$ be the fundamental dominant weights of the simple Lie algebra of type $D_n$. If $\lambda=\sum\limits_{i=1}^{n}a_i\lambda_i$ is the highest weight of an irreducible component of $T_X^{\otimes r}$. Then
  $$
  a_1=r-\sum_{i=2}^{n-2}a_i-\frac{1}{2}(a_{n-1}+a_n).
  $$
\end{lemma}

Combining Lemma \ref{decomp}, Lemma \ref{oddQ} and Lemma \ref{evenQ}, we can get the irreducible decompositions of symmetric products of the tangent bundles of hyperquadrics.
\begin{proposition}\label{dQ}
    Let $X$ be a smooth hyperquadric of dimension $\geq 3$. Then we have the following irreducible decomposition:
   $$
   Sym^r T_X=\left \{
   \begin{aligned}
    &\bigoplus_{i=0}^{\lfloor\frac{r}{2}\rfloor}E_{(r-2i)2\lambda_2+2i\lambda_1}\ , \ \ \ \ \ \ \ X=Q^3;\\
    &\bigoplus_{i=0}^{\lfloor\frac{r}{2}\rfloor}E_{(r-2i)(\lambda_2+\lambda_3)+2i\lambda_1}\ , \ \ X=Q^4;\\
    &\bigoplus_{i=0}^{\lfloor\frac{r}{2}\rfloor}E_{(r-2i)\lambda_2+2i\lambda_1}\ ,\ \ \ \ \ \ \ \ \text{otherwise}.
   \end{aligned}
   \right.
   $$
\end{proposition}
\begin{proof}
  Let $\gamma_i(i=1,2,..,n-1)$ be the fundamental dominant weights of $B_{n-1}$ or $D_{n-1}$. Notice that there are natural embeddings from the weight lattices of $B_{n-1}$ and $D_{n-1}$ into the weight lattices of $B_{n}$ and $D_{n}$ defined by $\gamma_i\mapsto \lambda_{i+1}$, respectively. Note that the highest weight of $T_X$ is $\lambda_2$ $($resp. $2\lambda_2,\lambda_2+\lambda_3)$ if $dim(X)>4$ $($resp. $dim(X)=3,dim(X)=4)$. If $dim(X)>4$, Lemma \ref{decomp} implies that
  $$
  Sym^r T_X=\bigoplus_{i=0}^{\lfloor\frac{r}{2}\rfloor}E_{(r-2i)\lambda_2+a_i\lambda_1},
  $$
  where each $a_i$ is some integer. Lemma \ref{oddQ} and Lemma \ref{evenQ} indicate that $a_i=2i$. If $X=Q^3$, by \cite[Exercise 11.14]{fulton2013representation} or by using the plethysm in Example \ref{plethysmexample}$(a)$, we have the following irreducible decomposition of $sl_2$-modules:
  $$
  Sym^rV_{2\gamma_1}=Sym^r(Sym^2V_{\gamma_1})=\bigoplus_{i=0}^{\lfloor\frac{r}{2}\rfloor}V_{(2r-4i)\gamma_1}.
  $$
  Hence the result follows from Lemma \ref{oddQ}. If $X=Q^4$. Note that $Q^4\simeq Gr(2,4)$ and the simple Lie algebra of type $D_3$ is exactly the simple Lie algebra of type $A_3$, hence the result follows from Remark \ref{decompositionGrass}.
\end{proof}

By the Borel-Weil-Bott's theorem, we have the following result.
\begin{theorem}
Let $X=Q^{m}(m\geq 3)$ be a smooth hyperquadric. Let $r$ be a positive integer. Then $H^p(X,Sym^r T_X\otimes \mathcal{O}_X(-d))\neq 0$ if and only if $p,r,d$ satisfy one of the following conditions:
\begin{enumerate}[$(a)$]
  \item $p=0$ and $d\leq 2\lfloor\frac{r}{2}\rfloor$;
  \item $p=1$ and $2\leq d\leq r+1$;
  \item $p=m-1$ and $r+m-1\leq d\leq 2r+m-2$;
  \item $p=m$ and $d\geq 2r+m-2\lfloor\frac{r}{2}\rfloor$.
 \end{enumerate}
 In particular,
  $$
   H^0(X,Sym^r T_X \otimes \mathcal{O}_X(-d))\neq 0 \Leftrightarrow \lfloor\frac{r}{2}\rfloor\geq \frac{d}{2}.
   $$
\end{theorem}
\begin{proof}
  We only consider the case that $X\simeq Q^{2n-2}(n\geq 4)$. For any simple root $\alpha$ of $\mathfrak{so}(2n)$, we use $\sigma_{\alpha}$ to represent the reflection with respective to the hyperplane orthogonal to $\alpha$. For $0\leq i\leq \lfloor\frac{r}{2}\rfloor$, let
  $$
  \begin{aligned}
  \xi_i&=(r-2i)\lambda_2+2i\lambda_1-d\lambda_1+\delta\\
       &=(2i+1-d)\lambda_1+(r-2i+1)\lambda_2+\lambda_3+...+\lambda_n.
  \end{aligned}
  $$
Then $H^0(X,Sym^r T_X\otimes \mathcal{O}_X(-d))\neq 0$ if and only if there exists $i$ such that $\xi_i$ is dominant, hence if and only if $2\lfloor\frac{r}{2}\rfloor\geq d$.

If there exists $i$ such that $\xi_i$ is regular of index $1$, then $2i+1-d<0$ and
$$
\sigma_{\alpha_1}(\xi_i)=(d-2i-1)\lambda_1+(r+2-d)\lambda_2+\lambda_3+...+\lambda_n
$$
is dominant. Hence $H^1(X,Sym^r T_X\otimes \mathcal{O}_X(-d))\neq 0$ if and only if $2\leq d\leq r+1$.

If $d\leq r+2$, then $\xi_i$ is singular or regular of index $\leq 1$. Now assume $d\geq r+3$. Then, for $2\leq p\leq n-3$, one can show that
$$
\begin{aligned}
&\ \ \ \sigma_{\alpha_p}\sigma_{\alpha_{p-1}}...\sigma_{\alpha_1}(\xi_i)\\
&=(r+1-2i)\lambda_1+\lambda_2+...+\lambda_{p-1}+(d-r-p)\lambda_p+(r+p+1-d)\lambda_{p+1}\\
&\ \ \ \ \ \ \ \ \ \ \ \ \ \ \ \ \ \ \ \ \ \ \ \ \ \ \ \ \ \ \ \ \ \ \ \ \ \  \ \ \ \ \ \ \ \ \ \ \ \ \ \  \ \ \ \ \ \ \ \ \ \ \ \ \ \ +\lambda_{p+2}+...+\lambda_n
\end{aligned}
$$
by induction on $p$ and
$$
\begin{aligned}
&\ \ \ \sigma_{\alpha_{n-2}}\sigma_{\alpha_{n-3}}...\sigma_{\alpha_1}(\xi_i)\\
&=(r+1-2i)\lambda_1+\lambda_2+...+\lambda_{n-3}+(d-r-n+2)\lambda_{n-2}+(r+n-1-d)\lambda_{n-1}\\
&\ \ \ \ \ \ \ \ \ \ \ \ \ \ \ \ \ \ \ \ \ \ \ \ \ \ \ \ \ \ \ \ \ \ \ \ \ \ \ \ \ \ \ \ \ \ \ \ \ \ \ \ \ \ \ \ \ \ \ \ \ \ \ \ \ \ \ \ \ \ \ \ \ \ \  +(r+n-1-d)\lambda_{n},\\
&\ \ \ \sigma_{\alpha_{n-1}}\sigma_{\alpha_{n-2}}...\sigma_{\alpha_1}(\xi_i)\\
&=(r+1-2i)\lambda_1+\lambda_2+...+\lambda_{n-3}+\lambda_{n-2}+(d-r-n+1)\lambda_{n-1}+(r+n-1-d)\lambda_{n}.
\end{aligned}
$$
Hence $\xi_i$ is singular for all $i$ if $r+2\leq d\leq r+n-1$. If $d\geq r+n$, then $\xi_i$ is singular or regular of index larger than $n-1$. Hence $H^p(X,Sym^r T_X\otimes \mathcal{O}_X(-d))=0$ for $2\leq p\leq n-1$. By the Serre duality and the fact that $T_X\simeq \Omega_X(2)$, we have
$$
H^p(X,Sym^r T_X\otimes \mathcal{O}_X(-d))\simeq H^{2n-2-p}(X,Sym^r T_X\otimes \mathcal{O}_X(d-2r-2n+2)).
$$
Hence we are done.
\end{proof}

\begin{remark}
Let $f(z_0,...,z_{n+1})=0$ be the defining equation of a smooth hyperquadric $Q^n(n\geq 2)$ in $\mathbb{P}^{n+1}$. One can deduce
\begin{center}
  $H^0(Q^n,Sym^{2r} T_{\mathbb{Q}^n} \otimes \mathcal{O}_{\mathbb{Q}^n}(-2r))\neq 0$
\end{center}
from the proof of \cite[Theorem B]{bogomolov2008symmetric}. In fact, any differential $\Omega \in Sym^{2r}[\mathbb{C}dz_0\oplus\mathbb{C}dz_1\oplus...\mathbb{C}dz_{n+1}]$ corresponds in a natural way to a homogeneous polynomial in $\mathbb{P}^{n+1}$. If it corresponds to the $r$-th power of $f$, then it will induce a nontrivial element in $H^0(X,Sym^{2r} \Omega_{\mathbb{Q}^n} \otimes \mathcal{O}_{\mathbb{Q}^n}(2r))\cong H^0(X,Sym^{2r} T_{\mathbb{Q}^n} \otimes \mathcal{O}_{\mathbb{Q}^n}(-2r))$ by
\cite[Proposition 1.5]{bogomolov2008symmetric}.
\end{remark}

\begin{remark}
By Proposition \ref{dQ}, if $dim(X)\geq 5$, the fundamental highest weights of $T_X^*$ are $\lambda_2$ and $2\lambda_1$ with degrees 1 and 2, respectively; if $dim(X)=4$, the fundamental highest weights of $T_X^*$ are $\lambda_2+\lambda_3$ and $2\lambda_1$ with degrees 1 and 2, respectively; if $dim(X)=3$, the fundamental highest weights of $T_X^*$ are $2\lambda_2$ and $2\lambda_1$ with degrees 1 and 2, respectively.
\end{remark}

\begin{remark}
We use the notations in Lemma \ref{decomp}. Notice that the exterior products of the $V_{\gamma_1}$ can be calculated explicitly. Hence by Lemma \ref{oddQ} and Lemma \ref{evenQ}, we can obtain irreducible decompositions of exterior products of $T_X$. We list the results as follows.

If $X=Q^{2n-1}(n\geq 3)$, then
$$
\bigwedge^r T_X=\left\{
\begin{aligned}
&E_{\lambda_{r+1}+(r-1)\lambda_1},\ \ \ \ r\leq n-2;\\
&E_{2\lambda_n+(n-2)\lambda_1},\ \ \ \ \ r= n-1;\\
&E_{2\lambda_n+(n-1)\lambda_1},\ \ \ \ \ r=n;\\
&E_{\lambda_{2n-r}+(r-1)\lambda_1},\ \ \ n+1\leq r\leq 2n-1.
\end{aligned}
\right.
$$
If $X=Q^{2n-2}(n\geq 4)$, then
$$
\bigwedge^r T_X=\left\{
\begin{aligned}
&E_{\lambda_{r+1}+(r-1)\lambda_1},\ \ \ \ \ \ \ \ \ \ \ \ \ \ \ \ \ \ \ \ \ \ \ \ \ \ \ \ \ r\leq n-3;\\
&E_{\lambda_n+\lambda_{n-1}+(n-3)\lambda_1},\ \ \ \ \ \ \ \ \ \ \ \ \ \ \ \ \ \ \ \ \ \ \ \ r= n-2;\\
&E_{2\lambda_n+(n-2)\lambda_1}\oplus E_{2\lambda_{n-1}+(n-2)\lambda_1},\ \ \ \ \ \ \ r=n-1;\\
&E_{\lambda_{n-1}+\lambda_{n}+(n-1)\lambda_1},\ \ \ \ \ \ \ \ \ \ \ \ \ \ \ \ \ \ \ \ \ \ \ \ \  r=n;\\
&E_{\lambda_{2n-1-r}+(r-1)\lambda_1},\ \ \ \ \ \ \ \ \ \ \ \ \ \ \ \ \ \ \ \ \ \ \ \ \ \ \ n+1\leq r\leq 2n-2.
\end{aligned}
\right.
$$
If $X=Q^3$, then
$$
\bigwedge^r T_X=\left\{
\begin{aligned}
&E_{2\lambda_2},\ \ \ \ \ \ \ \ \ r=1;\\
&E_{2\lambda_2+\lambda_1},\ \ \ \ \ r=2;\\
&E_{3\lambda_1},\ \ \ \  \ \ \ \ \ \ r=3.
\end{aligned}
\right.
$$
If $X=Q^4$, then
$$
\bigwedge^r T_X=\left\{
\begin{aligned}
&E_{\lambda_2+\lambda_3},\ \ \ \ \ \ \ \ \ \ \ \ \ \ \ \ \ \ \ \ \ r=1;\\
&E_{2\lambda_2+\lambda_1}\oplus E_{2\lambda_3+\lambda_1},\ \ \ \ \ \ r=2;\\
&E_{\lambda_2+\lambda_3+2\lambda_1},\ \ \ \ \ \ \ \ \ \ \ \ \ \ \ \ r=3;\\
&E_{4\lambda_1},\ \ \ \ \ \ \ \ \ \ \ \ \ \ \ \ \ \ \ \ \ \ \ \ \ r=4.
\end{aligned}
\right.
$$
From this we can compute the cohomology $H^p(X,\bigwedge\limits^rT_X(-d))$, but we leave the details to the readers.
\end{remark}

\subsection{Lagrangian Grassmannian}
Let $X=Lag(n,2n)\cong Sp(2n)/P_n$. The semisimple part $S$ of $P_n$ is $SL(n)$. The dual of the universal bundle of X is $E_{\lambda_1}$ with highest weight $\lambda_1$. All the weights of $E_{\lambda_1}$ are
$$
\lambda_1,\ \lambda_2-\lambda_1,...\ ,\lambda_n-\lambda_{n-1}.
$$
Note that the tangent bundle $T_X$ is isomorphic to $Sym^2E_{\lambda_1}$, thus we have the following decompositions thanks to the plethysm in Example \ref{plethysmexample}:
$$
Sym^r T_X=Sym^r(Sym^2 E_{\lambda_1})=\bigoplus_\mu (\mathbb{S}_\mu E_{\lambda_1}),
$$
summed over all even partitions $\mu$ of $2r$(i.e.each parts of $\mu$ is even), and
$$
\bigwedge^r T_X=\bigwedge^r(Sym^2 E_{\lambda_1})=\bigoplus_\mu (\mathbb{S}_{\mu'} E_{\lambda_1}),
$$
summed over all the partitions $\mu$ of the form $(c_1-1,...,c_p-1|c_1,...,c_p)$, where $c_1>...>c_p>0$ and $c_1+...+c_p=r$.

By the similar arguments as those for Proposition \ref{highestweight}, we have the following result:
\begin{proposition}\label{weightsLag}
  Let $\mu=(\mu_1,\mu_2,...,\mu_n)$ be a partition. Then $\mathbb{S}_\mu E_{\lambda_1}$ is an irreducible homogeneous vector bundle with highest weight:
  \begin{equation}\nonumber
    \mu_1\lambda_1+\mu_2(\lambda_2-\lambda_1)+...+\mu_n(\lambda_n-\lambda_{n-1}).
  \end{equation}
\end{proposition}

\begin{remark}
Let
$$
\begin{aligned}
  &i_j=(\mu_j-\mu_{j+1})/2,\ \ \ j=1,...,n-1, \\
  &i_n=\mu_n/2.
\end{aligned}
$$
By Proposition \ref{weightsLag}, we can rewrite the irreducible decomposition of $Sym^r T_X$ as
$$
Sym^r T_X=\bigoplus_{i_1+2i_2+...+ni_n=r\atop i_1,...,i_n\geq 0}E_{i_1\cdot2\lambda_1+i_2\cdot2\lambda_2+...+i_n\cdot 2\lambda_n}.
$$
We can also obtain the irreducible decomposition by using Lemma \ref{linear-relation}. Hence the fundamental highest weights of $T_X^*$ are $2\lambda_1,...,2\lambda_n$ with degree $1,...,n$, respectively.
\end{remark}
By the Borel-Weil-Bott's theorem, we have the following result.
\begin{theorem}
 $H^0(X,Sym^r T_X \otimes \mathcal{O}_X(-d))\neq 0 \Leftrightarrow 2\lfloor\frac{r}{n}\rfloor\geq d$.
\end{theorem}

\subsection{Spinor varieties}

 Let $X=Spin(2n)/P_n(n\geq 4)$ be a Spinor variety. The semisimple part $S$ of $P_n$ is $SL(n)$ as well. The dual of the universal bundle of X is $E_{\lambda_1}$ with highest weight $\lambda_1$. All the weights of $E_{\lambda_1}$ are
 $$
 \lambda_1,\ \lambda_2-\lambda_1,...\ , \lambda_{n-2}-\lambda_{n-3},\ \lambda_n+\lambda_{n-1}-\lambda_{n-2},\ \lambda_n-\lambda_{n-1}.
 $$

 Note the tangent bundle $T_X$ is isomorphic to $\bigwedge\limits^2 E_{\lambda_1}$, hence we have the following decompositions thanks to the plethysm in Example \ref{plethysmexample}:
\begin{equation}\nonumber
Sym^r T_X=Sym^r(\bigwedge\limits^2 E_{\lambda_1})=\bigoplus_\mu (\mathbb{S}_{\mu '} E_{\lambda_1}),
\end{equation}
summed over all even partitions $\mu$ of $2r$, and
$$
\bigwedge^r T_X=\bigwedge^r(\bigwedge^2 E_{\lambda_1})=\bigoplus_\mu (\mathbb{S}_{\mu} E_{\lambda_1}),
$$
summed over all the partitions $\mu$ of the form $(c_1-1,...,c_p-1|c_1,...,c_p)$, where $c_1>...>c_p>0$ and $c_1+...+c_p=r$.

By the similar arguments as those for Proposition \ref{highestweight}, we have the following result:
\begin{proposition}\label{wSpinor}
  Let $\mu=(\mu_1,\mu_2,...,\mu_n)$ be a partition. Then $\mathbb{S}_\mu E_{\lambda_1}$ is an irreducible homogeneous vector bundle with highest weight:
  \begin{equation}\nonumber
    \mu_1\lambda_1+\mu_2(\lambda_2-\lambda_1)+...+\mu_{n-2}(\lambda_{n-2}-\lambda_{n-3})+\mu_{n-1}(\lambda_n+\lambda_{n-1}-\lambda_{n-2})+\mu_{n}(\lambda_n-\lambda_{n-1}).
  \end{equation}
\end{proposition}

\begin{remark}
Let $\mu$ be an even partition of $2r$ and $\mu'$ its conjugation. Let $t=\lfloor\frac{n}{2}\rfloor$ and
$$
\begin{aligned}
  &i_j=\mu_{2j}'-\mu_{2j+1}',\ \ \ j=1,...,t-1, \\
  &i_t=\mu_{n-1}'.
\end{aligned}
$$
By Proposition \ref{wSpinor}, we can rewrite the irreducible decomposition of $Sym^r T_X$ as :
$$
Sym^r T_X=
\left\{
\begin{aligned}
&\bigoplus_{i_1+2i_2...+ti_t=r\atop i_1,...,i_t\geq 0}E_{i_1\lambda_2+i_2\lambda_4+...+i_t(\lambda_{n-1}+\lambda_n)}\ \ \ n\ is\ odd;\\
&\bigoplus_{i_1+2i_2...+ti_t=r\atop i_1,...,i_t\geq 0}E_{i_1\lambda_2+i_2\lambda_4+...+i_t(2\lambda_n)}\ \ \ \ \ \ \ \ \ n\ is\ even.
\end{aligned}
\right.
$$
Hence if $n$ is odd, the fundamental highest weights of $T_X^*$ are $\lambda_2,\lambda_4,...,\lambda_{2t-2},\lambda_{n-1}+\lambda_n$ with degree $1,2,...,\lfloor\frac{n}{2}\rfloor$, respectively; if $n$ is even, the fundamental highest weights of $T_X^*$ are $\lambda_2,\lambda_4,...,\lambda_{2t-2},2\lambda_n$ with degree $1,2,...,\lfloor\frac{n}{2}\rfloor$, respectively. Note that we can also obtain the irreducible decomposition by using Lemma \ref{linear-relation}.
\end{remark}
By the Borel-Weil-Bott's theorem, we have the following result.
\begin{theorem}
\begin{enumerate}[$(1)$]
    \item If $n$ is odd, then
    $$
    H^0(X,Sym^r T_X \otimes \mathcal{O}_X(-d))\neq 0 \Leftrightarrow \lfloor\frac{2r}{n-1}\rfloor\geq d;
    $$
    \item If $n$ is even, then
    $$
    H^0(X,Sym^r T_X \otimes \mathcal{O}_X(-d))\neq 0 \Leftrightarrow 2\lfloor\frac{2r}{n}\rfloor\geq d.
    $$
   \end{enumerate}
\end{theorem}

\subsection{Cayley plane}
Let $X=\mathbb{OP}^2\cong E_6/P_1$ be the Cayley plane. The semisimple part of the Lie algebra of $P_1$ is the simple Lie algebra $\mathfrak{so}(10)$.

Assume that the fundamental dominant weights of $\mathfrak{so}(10)$ and $E_6$ are $\gamma_i(i=1,2,3,4,5)$ and $\lambda_i(i=1,2,3,4,5,6)$, respectively. There is a natural embedding of the weight lattice of $\mathfrak{so}(10)$ into the weight lattice of $E_6$ defined by $\gamma_i\mapsto \lambda_{7-i}$. As a $\mathfrak{so}(10)$-module, $T_X$ is isomorphic to $V_{\gamma_5}$. By Kac's classification of the multiplicity free actions, $V_{\gamma_5}$ is a multiplicity free $\mathfrak{so}(10)$-module and the irreducible decomposition of $Sym^r V_{\gamma_5}$ is given by the following.

\begin{lemma}\label{symE_6}
  We have the following irreducible decomposition of $\mathfrak{so}(10)$-module:
  $$
  Sym^rV_{\gamma_5}=\bigoplus_{i+2j=r \atop i,j\geq 0}V_{i\gamma_5+j\gamma_1}.
  $$
\end{lemma}
\begin{proof}
  See the table in page 13 in \cite{MR698847}.
\end{proof}

By Lemma \ref{linear-relation}, we have the following.

\begin{lemma}\label{a1}
  If $\lambda=\sum\limits_{i=1}^{6} a_i \lambda_i$ is the highest weight of an irreducible component of $T_X^{\otimes r}$, then
  $$
  a_1=\frac{3}{4}r-\frac{3}{4}a_2-\frac{5}{4}a_3-\frac{3}{2}a_4-a_5-\frac{1}{2}a_6.
  $$
\end{lemma}

Combining Lemma \ref{symE_6} and Lemma \ref{a1}, we can get the irreducible decomposition of $Sym^r T_X$.
\begin{proposition}
We have the following irreducible decomposition:
  $$
  Sym^r T_X=\bigoplus_{i+2j=r\atop i,j\geq 0} E_{i\lambda_2+j(\lambda_6+\lambda_1)}.
  $$
   Hence the fundamental highest weights of $T_X^*$ are $\lambda_2$ and $\lambda_6+\lambda_1$ with degree 1 and 2, respectively.
\end{proposition}
 
The Borel-Weil-Bott's theorem yields the following result.
\begin{theorem}
 $H^0(X,Sym^r T_X\otimes \mathcal{O}_X(-d))\neq 0$ if and only if  $\lfloor\frac{r}{2}\rfloor\geq d$.
\end{theorem}

\subsection{\boldmath{$E_7/P_7$}}
Let $X=E_7/P_7$. The semisimple part of $P_7$ is $E_6$. 

Assume that the fundamental dominant weights of $E_6$ and $E_7$ are $\gamma_i(i=1,2,3,4,5,6)$ and $\lambda_i(i=1,2,..,7)$, respectively. There is a natural embedding of the weight lattice of $E_6$ into the weight lattice of $E_7$ defined by $\gamma_i\mapsto \lambda_i$. As a $E_6$-module, $T_X$ is isomorphic to $V_{\gamma_1}$. The irreducible decomposition of $Sym^r V_{\gamma_1}$ is given by the following.

\begin{lemma}\label{symE_7}
We have the following irreducible decomposition of $E_6$-module:
  $$
  Sym^rV_{\gamma_1}=\bigoplus_{i+2j+3k=r \atop i,j,k\geq 0}V_{i\gamma_1+j\gamma_6+k\cdot 0}.
  $$
\end{lemma}
\begin{proof}
See the table in page 13 in \cite{MR698847}.
\end{proof}

By Lemma \ref{linear-relation}, we can get
\begin{lemma}\label{a7}
  If $\lambda=\sum\limits_{i=1}^{7} a_i \lambda_i$ is the highest weight of an irreducible component of $T_X^{\otimes r}$, then
  $$
  a_7=\frac{2}{3}r-\frac{2}{3}a_1-a_2-\frac{4}{3}a_3-2a_4-\frac{5}{3}a_5-\frac{4}{3}a_6.
  $$
\end{lemma}

Combining Lemma \ref{symE_7} and Lemma \ref{a7}, we can get the irreducible decomposition of $Sym^r T_X$.

\begin{proposition}
  We have the following the irreducible decomposition of $Sym^r T_X$:
$$
Sym^r T_X=\bigoplus_{i+2j+3k=r \atop i,j,k\geq 0}E_{i\lambda_1+j\lambda_6+k\cdot 2\lambda_7}
$$
and the fundamental highest weights of $T_X^*$ are $\lambda_1,\lambda_6,2\lambda_7$ with degree $1,2,3$, respectively.
\end{proposition}

By the Borel-Weil-Bott's theorem, we have the following result.
\begin{theorem}
$H^0(X,Sym^r T_X\otimes \mathcal{O}_X(-d))\neq 0$ if and only if  $2\lfloor\frac{r}{3}\rfloor\geq d$.

\end{theorem}
\section{proof of the main results}\label{pseudo threshold}
As a corollary of previous discussions, the fundamental highest weights of the dual of tangent bundles of IHSS have the following properties:
\begin{proposition}\label{fhw}
Let $X=G/P_k$ be an IHSS. The degrees of the fundamental highest weights of $T_X^*$ are all distinct, they are exactly $1,2,...,rk(X)$. Moreover, each fundamental highest weight is a dominant weight of $G$. Let $\lambda_k'$ be the highest weight of the dual of the representation $V_{\lambda_k}$ of $G$. Then $\lambda_k'$ is a fundamental dominant weight of $G$ and the fundamental highest weight of $T_X^*$ with the highest degree $rk(X)$ is $\lambda_k+\lambda_k'$. If we write each fundamental highest weight as a linear combination of the fundamental dominant weights of $G$, $\lambda_k+\lambda_k'$ is the only fundamental highest weight with the coefficient of $\lambda_k$ nonvanishing.
\end{proposition}

Theorem \ref{maincoh} can be deduced from irreducible decompositions of symmetric products of tangent bundles of IHSS and one can also use Proposition \ref{fhw} to give a uniform proof:

\begin{proof}[Proof of Theorem \ref{maincoh}]
Assume $t=rk(X)$. Let $\omega_1,\omega_2,...,\omega_t$ be the fundamental highest weights of $T_X^*$ with degree 1,2,..,t, respectively. Then by Proposition \ref{pvd},  we have the following irreducible decomposition of $Sym^r T_X$:
$$
Sym^r T_X=\bigoplus_{i_1+2i_2+...+ti_t=r\atop i_1,i_2,...,i_t\geq 0}E_{i_1\omega_1+...+i_t\omega_t}.
$$
By Proposition \ref{fhw}, each fundamental highest weight is dominant and $\omega_t=\lambda_k+\lambda_k'$ is the only fundamental highest weight with non-vanishing coefficient of $\lambda_k$.

If $V_{\lambda_k}$ is self-dual, then $\lambda_k=\lambda_k'$ and $\omega_t=2\lambda_k$. Hence by the Borel-Weil-Bott's theorem, we have
$$
H^0(X,Sym^r T_X \otimes \mathcal{O}_X(-d))\neq 0 \Leftrightarrow 2\lfloor\frac{r}{rk(X)}\rfloor\geq d.
$$
If $V_{\lambda_k}$ is not self-dual, then the coefficient of $\lambda_k$ in $\omega_t$ is 1 and hence
$$
H^0(X,Sym^r T_X \otimes \mathcal{O}_X(-d))\neq 0 \Leftrightarrow \lfloor\frac{r}{rk(X)}\rfloor\geq d.
$$
\end{proof}

\begin{proof}[Proof of Corollary \ref{corrank} and Theorem \ref{minimal embedding}]
Let $X=X_1\times X_2\times ...\times X_m$, where each $X_i$ is an IHSS. Let $\pi_i:X\rightarrow X_i$ be the natural projection. Then $T_X=\bigoplus\limits_i \pi_i^* T_{X_i}$ and $Sym^r T_X=\bigoplus\limits_{r_1+...+r_m=r}\bigotimes\limits_{i}Sym^{r_i}\pi^*T_{X_i}$. By K\"{u}nneth formula, we have
$$
H^0(X,Sym^r T_X\otimes \mathcal{O}_X(-1))=\bigoplus\limits_{r_1+...+r_m=r}\bigotimes\limits_{i}H^0(X_i,Sym^{r_i}T_{X_i}\otimes \mathcal{O}_{X_i}(-1)).
$$
More generally,
$$
H^p(X,Sym^r T_X\otimes \mathcal{O}_X(-1))=\bigoplus\limits_{r_1+...+r_m=r}\bigoplus\limits_{p_1+...+p_m=p}\bigotimes\limits_{i}H^{p_i}(X_i,Sym^{r_i}T_{X_i}\otimes \mathcal{O}_{X_i}(-1)).
$$
Hence we can reduce to the case that $X=G/P_k$ is an IHSS. Hence Corollary \ref{corrank} follows from Theorem \ref{maincoh}.
If $r=rk(X)$, then
$$
H^0(X,Sym^rT_X\otimes \mathcal{O}_X(-1))\simeq H^0(X,E_{\omega_t-\lambda_k})
=H^0(X,E_{\lambda_k'})\simeq H^0(X,\mathcal{O}_X(1))^*.
$$
For the statement in Theorem \ref{minimal embedding}(2), notice that
$$
i_1\omega_1+...+i_t\omega_t-\lambda_k+\delta
$$
is either singular or regular of index 0. Hence the result follows from the Borel-Weil-Bott's theorem.
\end{proof}

Now we want to extend Corollary \ref{corrank}  to general rational homogeneous spaces. First we give a simple proof of the following result in \cite[Corollary 4.4]{greb2020canonical}.
\begin{proposition}\label{bignessofRHS}
  Rational homogeneous spaces have big tangent bundles.
\end{proposition}
\begin{proof}
Let $X=G/P$ be a rational homogeneous space. Let $V=H^0(X,T_X)$. Consider the evaluation morphism of global sections $\varepsilon:\mathbb{P}(T_X)\rightarrow \mathbb{P}(V)$. It is a generically finite morphism and its Stein factorization is exactly
$$\mathbb{P}(T_X)\stackrel{f}\rightarrow Proj(\bigoplus\limits_{i\geq 0}H^0(X,Sym^i T_X))\rightarrow \mathbb{P}(V).$$
Hence $\mathbb{P}(T_X)$ and $Proj(\bigoplus\limits_{i\geq 0}H^0(X,Sym^i T_X))$ have the same dimension and $T_X$ is big.
\end{proof}

Now let X=$G/P$ be a rational homogeneous space and $L$ an ample line bundle on $X$. Notice that
\begin{center}
  $H^0(\mathbb{P}(T_X),\mathcal{O}_{\mathbb{P}(T_X)}(1))=H^0(X,T_X)\neq 0$,
\end{center}
hence we have an inclusion:
\begin{center}
  $\mathcal{O}_{\mathbb{P}(T_X)}(r)\otimes \pi^{*}L^{-1}\hookrightarrow \mathcal{O}_{\mathbb{P}(T_X)}(r+1)\otimes \pi^{*}L^{-1}$
\end{center}
for every integer $r$, where $\pi:\mathbb{P}(T_X)\rightarrow X$ is the natural projection. It follows that if $H^0(X,Sym^{r_0} T_X\otimes L^{-1})\neq 0$ for some positive integer $r_0$, then $H^0(X,Sym^{r} T_X\otimes L^{-1})\neq 0$ for all integer $r\geq r_0$. Hence we have
\begin{proposition}\label{r0}
  Let $X$ be a rational homogeneous space and $L$ an ample line bundle on $X$. Then there exists a positive integer $r_0$ such that
  \begin{center}
   $H^0(X,Sym^r T_X \otimes L^{-1})\neq 0$\ \ if and only if\ \ $r\geq r_0$.
\end{center}
\end{proposition}

The integer $r_0$ in the Proposition \ref{r0} depends on the ample line bundle $L$. However, if $X$ is an HSS and $L$ is the very ample line bundle which gives a minimal equivariant closed embedding of $X$ into a projective space, this integer $r_0$ is actually the rank of $X$. So It is natural to ask the following questions:
\begin{question}\label{question}
What is the positive integer $r_0$ in the Proposition \ref{r0} when $L=\mathcal{O}_X(1)$ is the very ample line bundle which gives a minimal equivariant closed embedding of $X$ into a projective space? Is there any geometric description of $r_0$?
\end{question}

{\bf Acknowledgements.}
The author is greatly indebted to his advisor Baohua Fu for illuminating discussions, guidance and revising this paper. The author is also very grateful to Jie Liu for helpful discussions and suggestions which improve this paper. The author also wants to thank Laurent Manivel and anonymous referees for valuable suggestions.

\section*{}


\bibliographystyle{alpha}
\bibliography{main}

\end{document}